\RequirePackage{fix-cm}
\documentclass[smallextended]{svjour3}       
\smartqed  

\journalname{JOTA}
\usepackage{booktabs}
\usepackage{amsmath}
\usepackage{amssymb,epsfig,multirow}
\usepackage{float}
\usepackage{color}
\usepackage{enumerate}
\usepackage{graphicx,subfigure,array}
\usepackage{algorithmicx,algorithm}
\numberwithin{equation}{section}
\newtheorem{lem}{Lemma}[section]

\newtheorem{thm}{Theorem}[section]
\newtheorem{cor}{Corollary}[section]

\newtheorem{conj}{Conjecture}[section]

\newtheorem{rem}{Remark}[section]

\begin{document}
\title{Trust-region and $p$-regularized subproblems:  local nonglobal minimum is the second smallest objective function value among all first-order stationary points
\thanks{This research was supported by the Beijing Natural Science Foundation, grant Z180005 and by the National Natural Science Foundation of China under grants 11822103, 11571029, and 11771056.}
}

\titlerunning{On local nonglobal minimiums of (TRS) and (p-RS)}        

\author{Jiulin Wang
 \and Mengmeng Song
\and Yong Xia}


\institute{
Jiulin Wang
 \and Mengmeng Song  \and Yong Xia \at
              LMIB of the Ministry of Education, School of Mathematical Sciences, Beihang University, Beijing, 100191, P. R. China \\
              \email{(wangjiulin@buaa.edu.cn (J. L. Wang); songmengmeng@buaa.edu.cn (M. M. Song); yxia@buaa.edu.cn (Y. Xia, corresponding author) )}
}

\date{Received: date / Accepted: date}

\maketitle
\begin{abstract}
The local nonglobal minimizer of  trust-region subproblem, if it exists, is shown to have the second smallest objective function value among all KKT points. This new property is extended to $p$-regularized subproblem. As a corollary, we show for the first time that finding the local nonglobal minimizer of Nesterov-Polyak subproblem corresponds to a generalized eigenvalue problem.

\end{abstract}
\section{Introduction}\label{sec1}
The trust-region subproblem (TRS) is to minimize a (possibly nonconvex) quadratic function over a Euclidean ball. It is an essential problem in each iteration of the trust region method for solving nonlinear programming problems \cite{Y15}. (TRS) has other applications in convex quadratic integer programming \cite{B15},
bounded linear regression, best rank-$1$ tensor approximation, and image deconvolution \cite{P20}.

Nonconvex (TRS) has the property of zero Lagrangian-duality gap, see the recent survey on hidden convexity \cite{X20}. There is a necessary and sufficient optimality condition for the global minimizer, established in the early 1980s \cite{G81,S82,M83}. In 1994, Mart\'{i}nez \cite{M94} proved that (TRS) has at most one local nonglobal minimizer. In 2020, Wang and Xia \cite{WX20} established the necessary and sufficient optimality condition for the local nonglobal minimizer.

Besides the second-order optimality conditions for local and global minimizers,
first-order optimality (KKT) conditions are deeply investigated for (TRS), see for example \cite{LPR98}. Gander \cite{G80} showed that the objective function values of KKT points  of linear least squares problem with a quadratic equality constraint  are monotonic in terms of their Lagrangian multipliers. In order to locate local solutions of the Celis-Dennis-Tapia (CDT) subproblem \cite{C85} which minimizes a quadratic function over the intersection of two ellipsoids, Chen and Yuan \cite{CY00} studied the monotonicity property for KKT points of (CDT) in a partial order relation of corresponding Lagrangian multipliers.

The $p$-regularized subproblem ($p$-RS) is an unconstrained optimization problem of minimizing a (possibly nonconvex) quadratic function with an additional $p$-th power regularization term of the norm of the variables \cite{GRT10,Hsia17}. In literature, if the regularization term is cubic, ($p$-RS) is known as Nesterov-Polyak subproblem \cite{N06}. For more references, we refer to \cite{W07,C11,CGT11,L20} and references therein. If the regularization term is quartic, ($p$-RS) reduces to the double-well potential optimization, which has particular applications in solid mechanics and quantum mechanics \cite{F17,X17}. 


In this paper, based on the monotonicity properties for first-order stationary points, we prove that the local nonglobal minimizer of (TRS), if it exists, has the second smallest objective function value among all KKT points. We then extend the new property to  ($p$-RS) and show that the local nonglobal minimizer of (TRS) has the smallest objective function value among all critical points. As a corollary,  
we show for the first time that the local nonglobal minimizer of the Nesterov-Polyak subproblem  could be founded by solving a generalized eigenvalue problem.

The rest of this paper is organized as follows. In Section \ref{sec2}, we review some properties for local minimizers of (TRS) and ($p$-RS). In Section \ref{sec3} and \ref{sec4}, we first show that monotonicity properties hold for first-order stationary points of (TRS) and ($p$-RS), respectively. As the main result of this paper, we prove that the local nonglobal minimums of (TRS) and ($p$-RS) are the second smallest objective function value among all first-order stationary points, respectively.  Conclusion and discussions are made in Section \ref{sec6}. 


\section{Preliminaries}\label{sec2}
In this section, we review some characterizations for local minimizers of (TRS) and ($p$-RS).
\subsection{Characterizations for local minimizers of (TRS)}\label{subsec2.1}
In this subsection, we briefly review some properties of local minimizers of nonconvex (TRS):
\begin{eqnarray}
(\rm TRS)~~&\min & f(x)=\frac{1}{2}x^TQx+c^Tx \nonumber\\
&{\rm s.t.}& x^Tx-1\le 0,\nonumber
\end{eqnarray}
where $Q \in \mathbb{R}^{n\times n}$ is not positive semidefinite and $c \in \mathbb{R}^{n}$.
Up to an eigenvalue decomposition, we can always assume that $Q$ is a diagonal matrix with $n$ diagonal elements being $(0>)~ \alpha_1\le\alpha_2\le\cdots\le\alpha_n$.

$(x;\lambda)\in \mathbb{R}^{n}\times\mathbb{R}$ is a KKT point of (TRS) if and only if it satisfies the following KKT system:
\begin{eqnarray}
&&Qx+c=-\lambda x, \label{kkt0}\\
&&\lambda(x^Tx-1)=0, \label{kkt1}\\
&&x^Tx-1\le0,~~\lambda\ge 0. \label{kkt2}
\end{eqnarray}
In case of $\lambda>0$,  \eqref{kkt1}-\eqref{kkt2} is equivalent to
\begin{equation}
x^Tx-1=0. \label{kkt3}
 \end{equation}
Suppose $-\lambda\not\in\{\alpha_1,\alpha_2,\cdots,\alpha_n\}$.
Then by substituting the $x$-solution of (\ref{kkt0}) into (\ref{kkt3}), we obtain that $\lambda$ must be  a zero point of the following  so-called {\it secular function}:
\begin{equation}
\varphi(\lambda)=\sum_{i=1}^n\frac{c_i^2}{(\alpha_i+\lambda)^2}-1.\label{secular}
\end{equation}
The first and second derivatives of $\varphi(\lambda)$ can be written as follows:
\begin{eqnarray}
&&\varphi'(\lambda)=-2\sum_{i=1}^n\frac{c_i^2}{(\alpha_i+\lambda)^3},\nonumber\\
&&\varphi''(\lambda)=6\sum_{i=1}^n\frac{c_i^2}{(\alpha_i+\lambda)^4}.\nonumber
\end{eqnarray}
If $\varphi(\lambda)$ has zero points, then $c\neq0$ and hence $\varphi''(\lambda)>0$ on each nonempty interval $(-\alpha_{i+1},-\alpha_{i})$ for $i\in\{1,2,\cdots,n-1\}$. It follows that, on each interval,
$\varphi(\lambda)$ is strongly convex and thus has at most two zero points.

In the early 1980s, necessary and sufficient conditions for global minimizers of (TRS) are established by Gay \cite{G81}, Sorensen \cite{S82}, Mor\'{e} and Sorensen \cite{M83}.
\begin{lem}[\cite{G81}, Theorem 2.1]\label{lem:g}
$x^*$ is a global minimizer of (TRS) if and only if ${x^*}^Tx^*=1$ and there is a $\lambda^*\in \mathbb{R}$ such that (\ref{kkt0}) holds and
\[
\lambda^*\ge -\alpha_1(>0).
\]
\end{lem}
In 1994, Mart\'{\i}nez \cite{M94} identified two cases where
the local nonglobal minimizer of (TRS) does not exist.
\begin{lem}[\cite{M94}, Lemmas 3.2, 3.3] \label{lem:l0}
Suppose either $\alpha_1=\alpha_2$ or $c_1=0$, there is no local nonglobal minimizer of (TRS).
\end{lem}
In the same paper, Mart\'{\i}nez \cite{M94} proved that (TRS) has at most one local nonglobal minimizer based on the following detailed characterization.
\begin{lem}[\cite{M94}, Theorem 3.1(i)] \label{lem:l1}
If $\underline{x}$ is a local nonglobal minimizer of (TRS), then there is a nonnegative $\underline{\lambda}\in(-\alpha_2,-\alpha_1)$ such that (\ref{kkt0}) holds, and $\underline{\lambda}$ is a zero point of $\varphi(\lambda)$ and
\[
\varphi'(\underline{\lambda})\ge 0.
\]
\end{lem}
Recently, Lemma \ref{lem:l1} has been updated by Wang and Xia \cite{WX20}.
\begin{lem}[\cite{WX20}, Theorem 3.1] \label{lem:l10}
$\underline{x}$ is a local nonglobal minimizer of (TRS),  if and only if there is a nonnegative $\underline{\lambda}\in(-\alpha_2,-\alpha_1)$ such that (\ref{kkt0}) holds, and $\underline{\lambda}$ is a zero point of $\varphi(\lambda)$ and
\[
\varphi'(\underline{\lambda})> 0.\nonumber
\]
\end{lem}
In 1998, Lucidi et al. \cite{LPR98} proved that the strict complementarity condition holds at the local nonglobal minimizer of (TRS).
\begin{lem}[\cite{LPR98}, Proposition 3.5] \label{lem:l2}
If $\underline{x}$ is a local nonglobal minimizer of (TRS), the corresponding Lagrangian multiplier is $\underline{\lambda}$, then (\ref{kkt1}) holds with
\[
\underline{x}^T\underline{x}-1=0,~~\underline{\lambda}>0.
\]
\end{lem}
In 2017, Adachi et al. \cite{A17} observed the following result.
\begin{lem}[\cite{A17}, Lemma 3.1] \label{lem:l3}
If $(x;\lambda)$ is a KKT point of (TRS) satisfying
\eqref{kkt0} and \eqref{kkt3}, then we have
\[
{\rm det}~M_1(\lambda)=0,
\]
where
\begin{eqnarray}
&&M_1(\lambda):=\left[ {\begin{array}{*{20}{c}}
-I& Q\\
Q&-cc^T
\end{array}} \right]-\lambda
\left[ {\begin{array}{*{20}{c}}
0& -I \\
-I &0
\end{array}} \right].\nonumber
\end{eqnarray}
That is, $\lambda$ is a real generalized eigenvalue of $M_1(\lambda)$.
\end{lem}
Let $(x;\lambda)$ satisfying \eqref{kkt0}. We can rewrite $M_1(\lambda)$ as
\[
M_1(\lambda)=\left[ {\begin{array}{*{20}{c}}
I& {} \\
{}&Q+\lambda I
\end{array}} \right]
\left[ {\begin{array}{*{20}{c}}
-I& I \\
I&-xx^T
\end{array}} \right]
\left[ {\begin{array}{*{20}{c}}
I& {} \\
{}&Q+\lambda I
\end{array}} \right],
\]
Then,
\begin{eqnarray}
{\rm det}~M_1(\lambda)=(-1)^n ({\rm det}(Q+\lambda I))^2(1-x^Tx).\label{M11}
\end{eqnarray}
The following result follows from Lemma \ref{lem:l3} and \eqref{M11}.
\begin{cor}[\cite{S17}, Corollary 3]\label{cor:eig1}
The set of real generalized eigenvalues of $M_1(\lambda)$ is nonempty. Moreover, if ${\rm det}~M_1(\lambda)=0$, then either $-\lambda$ is an eigenvalue of $Q$ or $\lambda$ is a zero point of $\varphi(\lambda)$.
\end{cor}

\subsection{Characterizations for local minimizers of ($p$-RS)}\label{subsec2.2}
In this subsection, we review some characterizations for local minimizers of ($p$-RS):
\begin{eqnarray}
(p\rm-RS)~~&\mathop{\min}\limits_{x\in \mathbb{R}^{n}} & \left\{g(x)=\frac{1}{2}x^TQx+c^Tx+\frac{\sigma}{p}\|x\|^p\right\}, \nonumber
\end{eqnarray}
where $\sigma>0$, $p>2$, $\|\cdot\|$ is the Euclidean norm, and $Q,c$ have the same definitions as in (TRS). The global minimum of ($p$-RS) is attainable, even if $Q$ is not positive semidefinite. Actually, $g(x)$ is coercive, i.e., 
\[
\lim_{\|x\|\rightarrow +\infty}g(x)=+\infty,
\]
since $p>2$ and then  
\[
\lim_{\|x\|\rightarrow +\infty}\frac{\frac{1}{2}x^TQx+c^Tx}{ \|x\|^p}=0. 
\]
We call
$x\in \mathbb{R}^{n}$ a critical point of ($p$-RS) if and only if it satisfies the following first order necessary condition for local minimizers of ($p$-RS):
\begin{eqnarray}
&&\nabla g(x)=(Q+\sigma\|x\|^{p-2}I)x+c=0, \label{nc1}
\end{eqnarray}
where $\nabla g$ denotes the gradient of $g(\cdot)$. Let
\begin{eqnarray}
&&t=\|x\|^{p-2}, \label{t1}
\end{eqnarray}
then (\ref{nc1}) can be rewritten as
\begin{eqnarray}
&&(Q+\sigma tI)x+c=0. \label{nc2}
\end{eqnarray}
Suppose $-\sigma t\not\in\{\alpha_1,\alpha_2,\cdots,\alpha_n\}$. Then by substituting the $x$-solution of (\ref{nc2}) into (\ref{t1}), we obtain that $t$ must be a zero point of the following secular function:
\begin{equation}
h(t)=\sum_{i=1}^n\frac{c_i^2}{(\sigma t+\alpha_i)^2}-t^{\frac{2}{p-2}}.\label{h0}
\end{equation}
In the following, we make an assumption that $x\neq 0$, then it holds that $c\neq0$ from (\ref{nc1}). Under this assumption, $h(t)$ has the same zeros as
\begin{equation}
p(t)={\rm log}\left(\sum_{i=1}^n\frac{c_i^2}{(\sigma t+\alpha_i)^2}\right)-\frac{2}{p-2}{\rm log}(t).\nonumber
\end{equation}
Hsia et al. \cite{Hsia17} proved that $p''(t)>0$ on each nonempty interval $\left(-\frac{\alpha_{i+1}}{\sigma},-\frac{\alpha_{i}}{\sigma}\right)$ for $i\in\{1,2,\cdots,n-1\}$.

The first derivatives of $h(t)$ and $p(t)$ can be written as follows:
\begin{eqnarray}
&&h'(t)=-\sum_{i=1}^n\frac{2\sigma c_i^2}{(\sigma t+\alpha_i)^3}-\frac{2}{(p-2)t}\cdot t^{\frac{2}{p-2}},\label{h1}\\
&&p'(t)=-\frac{\sum_{i=1}^n\frac{2\sigma c_i^2}{(\sigma t+\alpha_i)^3}}{\sum_{i=1}^n\frac{c_i^2}{(\sigma t+\alpha_i)^2}}-\frac{2}{(p-2)t}.\label{pt}
\end{eqnarray}
If $\hat{t}$ is a zero point of $h(t)$, according to (\ref{h0}), (\ref{h1}) and (\ref{pt}), we obtain
\begin{eqnarray}
&&p'(\hat{t})=h'(\hat{t})\cdot \hat{t}^{{\frac{-2}{p-2}}}.\label{pt1}
\end{eqnarray}
The following detailed characterizations of local minimizers of ($p$-RS) are all due to Hsia et al.
\cite{Hsia17}.
\begin{lem}[\cite{Hsia17}, Theorem 2.2]\label{lem:p1}
$x^*$ is a global minimizer of ($p$-RS) for $p > 2$ if and only if it is a critical point satisfying (\ref{nc1}) and
\[
\sigma\|x^*\|^{p-2}+\alpha_1\ge0.
\]
Moreover, the $\ell_2$ norms of all the global minimizers are equal.
\end{lem}
\begin{lem}[\cite{Hsia17}, Lemma 3.1]\label{lem:p2}
Suppose $\underline{x}$ is a local non-global minimizer of ($p$-RS) for $p>2$. It holds that $\underline{x}_1\neq 0$, $\alpha_1<\alpha_2$, and
\[
\sigma\|\underline{x}\|^{p-2}+\alpha_2>0.
\]
\end{lem}
\begin{lem}[\cite{Hsia17}, Theorem 3.2]\label{lem:p3}
The point $\underline{x}$ is a local nonglobal minimizer of ($p$-RS) for $p>2$ if and only if $\underline{x}$ is a critical point satisfying (\ref{nc1}) and $\underline{t}$ {\rm(=$\|\underline{x}\|^{p-2}$)} is a root of
\[
h(t)=0 ,~t\in \left({\rm max}\left\{-\frac{\alpha_2}{\sigma},0\right\},-\frac{\alpha_1}{\sigma} \right)
\]
such that $h'(\underline{t})>0$.
\end{lem}
Motivated by the observation that the Lagrange multiplier of each KKT point of (TRS) satisfying \eqref{kkt0} and \eqref{kkt3} is a generalized eigenvalue of $M_1(\lambda)$ \cite{A17}, Lieder \cite{L20} pointed out the following result holds for the Nesterov-Polyak subproblem (i.e., ($p$-RS) for $p=3$).
\begin{lem}[\cite{L20}, Lemma 3.1]\label{lem:eig}
If $x$ is a critical point of the cubic regularization, then $t~(=\|x\|)$ is a real generalized eigenvalue of
\begin{eqnarray}
&&M_2(t)=\left[ {\begin{array}{*{20}{c}}
0& 0& 0& c^T \\
0&-\sigma I&0& Q \\
0&0&-\sigma& 0 \\
c&Q&0&0
\end{array}} \right]-t
\left[ {\begin{array}{*{20}{c}}
0& 0& -\sigma& 0 \\
0&0&0& -\sigma I \\
-\sigma&0&0& 0 \\
0&-\sigma I&0&0
\end{array}} \right].\nonumber
\end{eqnarray}
That is, $t$ is a real root of ${\rm det}~M_2(t)=0$.
\end{lem}
Let $(x;t)$ satisfying (\ref{nc2}). We can reformulate $M_2(t)$ as
\[
M_2(t)=\left[ {\begin{array}{*{20}{c}}
I& {}& {}& {} \\
{}&I&{}& {} \\
{}&{}&1& {} \\
{}&{}&{}& Q+\sigma tI
\end{array}} \right]
\left[ {\begin{array}{*{20}{c}}
0& 0& \sigma t& -x^T\\
0&-\sigma I&0& I \\
\sigma t&0&-\sigma & 0 \\
-x &I&0&0
\end{array}} \right]
\left[ {\begin{array}{*{20}{c}}
I& {}& {}& {} \\
{}&I&{}& {} \\
{}&{}&1& {} \\
{}&{}&{}& Q+\sigma tI
\end{array}} \right],
\]
Then, by Schur complement,
\begin{eqnarray}
{\rm det}~M_2(t)=(-1)^{n+1}\sigma^2\cdot({\rm det}(Q+\sigma t I))^2(t^2-x^Tx).\label{M22}
\end{eqnarray}
The following result follows from Lemma \ref{lem:eig} and \eqref{M22}.
\begin{cor}\label{cor:eig2}
The set of real generalized eigenvalues of $M_2(t)$ is nonempty. Moreover, if ${\rm det}~M_2(t)=0$, then one of the following cases happens: (1) $t=0$; (2) $-\sigma t$ is an eigenvalue of $Q$; (3) $t$ is a zero point of $p(t)$.
\end{cor}

\section{On local nonglobal minimum of (TRS)}\label{sec3}
We first prove a monotonicity property for KKT points of (TRS).
\begin{lem}\label{lem:kkt1}
Let $(x_1;\lambda_1)$ and $(x_2;\lambda_2)$ be two KKT points for (TRS). Then
\[
\lambda_2\ge(>)\lambda_1 ~ \Longrightarrow~ f(x_2)\le (<) f(x_1).
\]
\end{lem}
\begin{proof}
According to (\ref{kkt0}) and the definitions of $(x_1;\lambda_1)$ and $(x_2;\lambda_2)$, we have
\begin{eqnarray}
&& Qx_1+c=-\lambda_1 x_1,~Qx_2+c=-\lambda_2 x_2. \label{x01}
\end{eqnarray}
By multiplying $x_1^T$ and $x_2^T$ to both sides of the two equalities in (\ref{x01}), respectively, we have
\begin{eqnarray}
&& x_1^TQx_1+x_1^Tc=-\lambda_1 x_1^Tx_1,~~x_1^TQx_2+x_1^Tc=-\lambda_2 x_1^Tx_2, \nonumber \\
&& x_2^TQx_2+x_2^Tc=-\lambda_2 x_2^Tx_2,~~x_2^TQx_1+x_2^Tc=-\lambda_1 x_2^Tx_1. \nonumber
\end{eqnarray}
Then, it holds that
\begin{eqnarray} 2(f(x_2)-f(x_1))&=&(x_2^TQx_2+x_2^Tc)-(x_1^TQx_1+x_1^Tc)+x_2^Tc-x_1^Tc \nonumber\\
&=&(-\lambda_2 x_2^Tx_2)-(-\lambda_1 x_1^Tx_1)+(-\lambda_1 x_2^Tx_1)-(-\lambda_2 x_1^Tx_2)\nonumber\\
&=&(\lambda_2-\lambda_1)x_1^Tx_2+\lambda_1x_1^Tx_1-\lambda_2x_2^Tx_2\label{l00}\\
&\le&(\lambda_2-\lambda_1)\cdot 1+\lambda_1x_1^Tx_1-\lambda_2x_2^Tx_2\label{l01}\\
&=&\lambda_1(x_1^Tx_1-1)-\lambda_2(x_2^Tx_2-1)=0, \label{l02}
\end{eqnarray}
where (\ref{l01}) follows from Cauchy-Schwarz inequality and (\ref{kkt2}), more precisely,
\begin{equation}
x_1^Tx_2\le \|x_1\|\cdot\|x_2\|\le 1,\label{CS}
\end{equation}
the last equality in (\ref{l02}) holds due to the complementarity condition (\ref{kkt1}). 

The equality $f(x_2)=f(x_1)$ holds true if and only if the inequality in (\ref{l01}) holds as an equality, that is
\begin{equation}
(\lambda_2-\lambda_1)(x_1^Tx_2-1)=0.\label{lbd}
\end{equation}
In case of $\lambda_2>\lambda_1$,  (\ref{lbd}) is equivalent to
\begin{equation}
x_1^Tx_2=1. \label{lbd2}
\end{equation}
It implies from \eqref{CS} and \eqref{lbd2} that
\begin{equation}
x_1=x_2\neq 0. \label{lbd3}
\end{equation}
Combining the assumption $\lambda_2>\lambda_1$ with  \eqref{kkt0} and \eqref{lbd3}
yields a contradiction. Therefore, $f(x_2)=f(x_1)$ if and only if $\lambda_2=\lambda_1$. The proof is complete.
\end{proof}
\begin{rem}
Let $(x_1;\lambda_1)$ and $(x_2;\lambda_2)$ be two KKT points for (TRS) satisfying $\|x_1\|=\|x_2\|$, since
\[
x_1^Tx_1-x_1^Tx_2=x_2^Tx_2-x_1^Tx_2=\frac{1}{2}\left\|x_1-x_2\right\|^2,
\]
it follows from \eqref{l00} that
\[
f(x_2)-f(x_1)=\frac{\lambda_1-\lambda_2}{4}\|x_1-x_2\|^2,
\]
which is presented in \cite{G80} for the relaxed least square problem.
\end{rem}

As a corollary of Lemma \ref{lem:kkt1}, the following result holds.

\begin{cor}\label{cor0}
If $(x^*;\lambda^*)$ is KKT point of (TRS), $x^*$ is a global minimizer, then $\lambda^*$ is the largest Lagrangian multiplier among all KKT points. Moreover, all the global minimizers of (TRS) share the same Lagrangian multiplier $\lambda^*$.
\end{cor}

Besides, according to Lemma \ref{lem:kkt1}, Lemma \ref{lem:g}, Lemma \ref{lem:l3} and Corollary \ref{cor:eig1}, we can recover the following result for the global minimizer of (TRS).

\begin{cor}[\cite{A17}, Theorem 3.2]\label{cor1}
If $x^*$ is a global minimizer of (TRS), $\lambda^*$ is the corresponding Lagrangian multiplier, then $\lambda^*$ is equal to the largest real generalized eigenvalue of $M_1(\lambda)$.
\end{cor}

We now present the main result of this section.
\begin{thm}\label{thm:kkt2}
Let $(x_1;\lambda_1)$ and $(x_2;\lambda_2)$ be any two different KKT points for (TRS).
If $x_1$ is a local nonglobal minimizer of (TRS) and $x_2$ is not a global minimizer of (TRS), then $f(x_1)<f(x_2)$.
\end{thm}
\begin{proof}
According to Lemma \ref{lem:kkt1}, it is sufficient to prove $\lambda_2<\lambda_1$.

If $x_2^Tx_2<1$, then (\ref{kkt1}) implies that $\lambda_2=0$. According to Lemma \ref{lem:l2}, it holds that $\lambda_1>0$. Then, we obtain
\[
\lambda_2=0<\lambda_1.
\]

If $x_2^Tx_2=1$, then either $-\lambda_2\in\{\alpha_1,\alpha_2,\cdots,\alpha_n\}$
or
$\lambda_2$ is a zero point of $\varphi(\lambda)$ (\ref{secular}). Since $x_2$ is not a global
minimizer of (TRS), it follows from Lemma \ref{lem:g} that
\[
\lambda_2<-\alpha_1.
\]
According to Lemma \ref{lem:l1}, $\lambda_1$ is the largest zero point of $\varphi(\lambda)$ in $(-\alpha_2,-\alpha_1)$. Then we have
\[
\lambda_2<\lambda_1<-\alpha_1.\nonumber
\]
The proof is complete.
\end{proof}

According to Lemma \ref{lem:l0}, if there is a local nonglobal minimizer of (TRS), then we have $\alpha_1\neq \alpha_2$ and $c_1\neq 0$. Then, ${\rm det}~M_1(-\alpha_1)\neq 0$. Otherwise, it is not difficult to verify that $c_1=0$, which leads to a contradiction. Combining ${\rm det}~M_1(-\alpha_1)\neq 0$, the proof of Theorem \ref{thm:kkt2}, Lemma \ref{lem:g}, Lemma \ref{lem:l3} and Corollary \ref{cor:eig1} together, we can recover the following result.

\begin{cor}[\cite{S17}, Theorem 5]\label{cor2}
If $\underline{x}$ is a local nonglobal minimizer of (TRS), $\underline{\lambda}$ is the corresponding Lagrangian multiplier, then $\underline{\lambda}$ is equal to the second largest real generalized eigenvalue of $M_1(\lambda)$.
\end{cor}

\section{On local nonglobal minimum of ($p$-RS)}\label{sec4}
We first prove a monotonicity property for critical points of ($p$-RS).
\begin{lem}\label{lem:cp}
Let $x_1$ and $x_2$ be two critical points of ($p$-RS) for $p>2$. Then
\[
\|x_2\|\ge(>)\|x_1\| ~ \Longrightarrow~ g(x_2)\le (<) g(x_1).
\]
\end{lem}
\begin{proof}
According to (\ref{nc1}) and the definitions of $x_1$ and $x_2$, we have
\begin{eqnarray}
&&(Q+\sigma\|x_1\|^{p-2}I)x_1+c=0,
~~(Q+\sigma\|x_2\|^{p-2}I)x_2+c=0. \label{c01}
\end{eqnarray}
By multiplying $x_1^T$ and $x_2^T$ to both sides of the two equalities in (\ref{c01}), respectively, we have
\begin{eqnarray}
&& x_1^TQx_1+x_1^Tc+\sigma\|x_1\|^{p}=0,
~~x_1^TQx_2+x_1^Tc+\sigma\|x_2\|^{p-2}x_1^Tx_2=0, \nonumber \\
&& x_2^TQx_2+x_2^Tc+\sigma\|x_2\|^{p}=0,
~~x_2^TQx_1+x_2^Tc+\sigma\|x_1\|^{p-2}x_1^Tx_2=0. \nonumber
\end{eqnarray}
Then, $g(x_2)-g(x_1)$ is equal to
\begin{eqnarray} &&\frac{1}{2}((x_2^TQx_2+x_2^Tc)-(x_1^TQx_1+x_1^Tc))+\frac{1}{2}(x_2^Tc-x_1^Tc) +\frac{\sigma}{p}(\|x_2\|^p-\|x_1\|^p)
\nonumber\\
&=&(\frac{\sigma}{p}-\frac{\sigma}{2})(\|x_2\|^p-\|x_1\|^p)
+\frac{\sigma}{2}(\|x_2\|^{p-2}-\|x_1\|^{p-2})x_1^Tx_2.
\label{cp01}
\end{eqnarray}
Then by (\ref{cp01}), we have
\begin{eqnarray}
&&\|x_2\|=\|x_1\|~\Longrightarrow~g(x_2)=g(x_1).\label{eq}
\end{eqnarray}
Next, we show that
\begin{eqnarray}
&&\|x_2\|>\|x_1\|~\Longrightarrow~g(x_2)<g(x_1).\label{ineq}
\end{eqnarray}
If $x_1=0$ and $x_2\neq 0$, it follows from (\ref{cp01}) and $p>2$ that
\begin{eqnarray}
g(x_2)-g(x_1)=
&&(\frac{\sigma}{p}-\frac{\sigma}{2})\|x_2\|^p< 0.\label{eq0}
\end{eqnarray}
If $x_1\neq0$, let $\|x_2\|=\alpha\|x_1\|$, where $\alpha\ge 1$.
Substituting $\|x_2\|=\alpha\|x_1\|$ to (\ref{cp01}) yields that
\begin{eqnarray}
g(x_2)-g(x_1)&=&(\frac{\sigma}{p}-\frac{\sigma}{2})
(\alpha^p-1)\|x_1\|^p
+\frac{\sigma}{2}(\alpha^{p-2}-1)\|x_1\|^{p-2}\cdot x_1^Tx_2\nonumber\\
&\le&\frac{\sigma}{2}\|x_1\|^p\cdot
\left(\frac{2-p}{p}(\alpha^p-1)+\alpha^{p-1}-\alpha\right).\label{cp02}
\end{eqnarray}
where (\ref{cp02}) follows from Cauchy-Schwarz inequality, i.e.,
\[
x_1^Tx_2\le \|x_1\|\cdot\|x_2\|= \alpha\|x_1\|^2.
\]
Define
\[
\phi(\alpha)=\frac{2-p}{p}(\alpha^p-1)+\alpha^{p-1}-\alpha.
\]
The first and second derivatives of $\phi(\alpha)$ can be written as follows:
\begin{eqnarray}
\phi'(\alpha)&=&(p-1)\alpha^{p-2}-(p-2)\alpha^{p-1}-1,\nonumber\\
\phi''(\alpha)&=&(p-1)(p-2)(1-\alpha)\alpha^{p-3}.\label{phi2}
\end{eqnarray}
If $\alpha=1$, we can verify that
\begin{eqnarray}
&&\phi(1)=\phi'(1)=\phi''(1)=0.\label{phi3}
\end{eqnarray}
If $\alpha>1$, according to (\ref{phi2}) and (\ref{phi3}), we have
\begin{eqnarray}
&&\phi''(\alpha)<\phi''(1)=0~\Longrightarrow~
\phi'(\alpha)<\phi'(1)=0 \Longrightarrow~\phi(\alpha)<\phi(1)=0.\label{neq0}
\end{eqnarray}
Then (\ref{ineq}) holds due to (\ref{eq0}) and (\ref{neq0}). Combining (\ref{eq}) with (\ref{ineq}), the proof is complete.
\end{proof}

If $p=3$, according to Lemma \ref{lem:cp}, Lemma \ref{lem:p1}, Lemma \ref{lem:eig} and Corollary \ref{cor:eig2}, we recover the following result.
\begin{cor}[\cite{L20}, Corollary 3.2]\label{cor3}
If $x^*$ is a global minimizer of ($p$-RS) for $p = 3$, then $t^*=\|x^*\|$ is equal to the largest real generalized eigenvalue of $M_2(t)$.
\end{cor}

Now we present the main result of this section.
\begin{thm}\label{thm:cp2}
Let $x_1$ and $x_2$ be any two different critical points of ($p$-RS) for $p>2$.
If $x_1$ is a local nonglobal minimizer of ($p$-RS) and $x_2$ is not a global minimizer, then $g(x_1)<g(x_2)$.
\end{thm}
\begin{proof}
Let $t_i=\|x_i\|^{p-2}$ for $i=1,2$. According to Lemma \ref{lem:cp}, it is sufficient to prove $\|x_2\|<\|x_1\|$ or equivalently, $t_2<t_1$.

We first prove that $c\neq0$. As $x_1$ is a local nonglobal minimizer of ($p$-RS), according to Lemma \ref{lem:p1}, we have
\begin{eqnarray}
\sigma\|x_1\|^{p-2}+\alpha_1<0.\label{lp1}
\end{eqnarray}
Then, it follows from (\ref{lp1}), (\ref{nc1}) and Lemma \ref{lem:p2} that $c_1\neq 0$, and hence $c\neq 0$.

Next, as $c\neq 0$, (\ref{nc1}) implies that $x_1,x_2\neq0$. Then, $p(t)$ is well defined. We can conclude that either $-\sigma t_2\in\{\alpha_1,\alpha_2,\cdots,\alpha_n\}$
or $t_2$ is a zero point of $p(t)$. Since $x_2$ is not a global minimizer of ($p$-RS), it follows from Lemma \ref{lem:p1} that
\[
t_2<-\frac{\alpha_1}{\sigma }.
\]
Since $p(t)$ is strongly convex for $t\in \left({\rm max}\left\{-\frac{\alpha_2}{\sigma},0\right\},-\frac{\alpha_1}{\sigma} \right)$, it has
at most two zero points in this interval. According to Lemma \ref{lem:p3} and (\ref{pt1}), we have
\[
p'(t_1)>0.
\]
Thus $t_1$ is the largest zero point of $p(t)$ in the interval $\left({\rm max}\left\{-\frac{\alpha_2}{\sigma},0\right\},-\frac{\alpha_1}{\sigma} \right)$. Then we have
\[
t_2<t_1<-\frac{\alpha_1}{\sigma}.
\]
The proof is complete.
\end{proof}

In Section \ref{sec3}, we have shown that if there is a local nonglobal minimizer of (TRS), then ${\rm det}~M_1(-\alpha_1)\neq 0$. Similarly, we can show that if there is a local nonglobal minimizer of ($p$-RS) for $p=3$, then ${\rm det}~M_2\left(-\frac{\alpha_1}{\sigma}\right)\neq 0$. Combing this result with the proof of Theorem \ref{thm:cp2}, Lemmas \ref{lem:p1}, \ref{lem:eig} and Corollary \ref{cor:eig2}, we give the following result.
\begin{cor}\label{cor4}
If $\underline{x}$ is the local nonglobal minimizer of the Nesterov-Polyak subproblem (i.e., ($p$-RS) for $p = 3$), then $\underline{t}=\|\underline{x}\|$ is equal to the second largest real generalized eigenvalue of $M_2(t)$.
\end{cor}

\section{Conclusion and discussions}\label{sec6}
This paper shows that there are monotonicity properties for first-order stationary points of trust-region subproblem (TRS) and $p$-regularized subproblem ($p$-RS). Based on these properties, we point out that all the global minimizers of (TRS) share the same Lagrangian multiplier and recover the result that finding a global minimizer of (TRS) (or ($p$-RS)) corresponds to a generalized eigenvalue problem. As the main contribution, we prove that the local nonglobal minimizer of (TRS) (or ($p$-RS)), if it exists, has the second smallest objective function value among all KKT points (critical points). For the Nesterov-Polyak subproblem (i.e., ($p$-RS) with $p=3$), we show for the first time that the local nonglobal minimizer, if it exists, could also be founded by solving a generalized eigenvalue problem.

Our main results for ($p$-RS) or (TRS) may fail to hold for more general optimization problems. 
Consider the problem of minimizing a univariate sextic function:
\[
s(x)=\frac{1}{6}x^6-\frac{21}{10}x^5
+\frac{57}{8}x^4-\frac{1}{12}x^3
-\frac{435}{32}x^2-\frac{297}{32}x.
\]
There are four critical points, denoted by $x_i$, $i=1,2,3,4$. 
As demonstrated in Fig. \ref{fig:1}, $x_1=-0.5$ is neither a local minimizer nor a local maximizer, $x_2=1.5$ is a global minimizer, $x_3=4.5$ is a local maximizer, and $x_4=5.5$ is a local nonglobal minimizer. It is observed from Fig. \ref{fig:1} that, the unique local nonglobal minimum is not the second smallest objective function value among all critical points. 
\begin{figure}[!htb]
\begin{center}
\centering
  \includegraphics[width=0.5\textwidth]{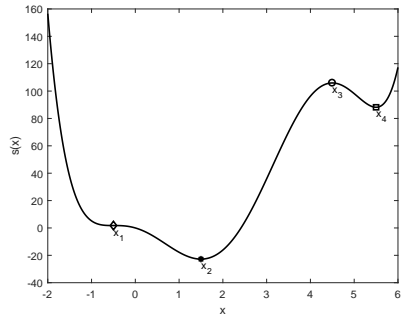}
\end{center}
\caption{Variation of $s(x)$ with four critical points satisfying $s(x_2)<s(x_1)<s(x_4)<s(x_3)$.}
\label{fig:1}
\end{figure}
Moreover, the above unconstrained minimization problem can be reformulated as the following nonconvex quadratic optimization with a quintic constraint:
\begin{eqnarray}
&\mathop{\min}\limits_{x,y\in \mathbb{R}} & xy \nonumber\\
&{\rm s.t.}& y=\frac{1}{6}x^5-\frac{21}{10}x^4
+\frac{57}{8}x^3-\frac{1}{12}x^2
-\frac{435}{32}x-\frac{297}{32}.\nonumber
\end{eqnarray}
One can verify that the constrained problem has four KKT points, corresponding the four critical points of $s(x)$. 

It seems that there is still a little room for possible extension. We conclude this paper with the following conjecture:
\begin{conj}
For ${\rm (CDT)}$  subproblem or nonconvex quadratic optimization with a quartic constraint, 
the smallest local nonglobal minimum has the second smallest objective function value among all KKT points.
\end{conj}


\begin{thebibliography}{99}

\bibitem{Y15}
{\sc Y. Yuan}, {\em Recent advances in trust region algorithms}, Math. Program., 151(1) (2015), pp. 249--281.

\bibitem{B15}
{\sc C. Buchheim, R. H\"{u}bner, A. Sch\"{o}bel}, {\em  Ellipsoid bounds for convex quadratic integer programming}, SIAM J. Optim., 25(2) (2015), pp. 741--769.

\bibitem{P20}
{\sc A. H. Phan, M. Yamagishi, D. Mandic, A. Cichocki}, {\em Quadratic programming over ellipsoids with applications to constrained linear regression and tensor decomposition}, Neural Comput. Appl.,  32(2020), pp.  7097--7120.

\bibitem{X20}
{\sc Y. Xia}, {\em  A survey of hidden convex optimization}, J. Oper. Res. Soc. China., 8(1) (2020), pp. 1--28.

\bibitem{G81}
{\sc D. M. Gay}, {\em Computing optimal locally constrained steps}, SIAM J. Sci. Stat. Comput., 2(2) (1981), pp. 186--197.

\bibitem{S82}
{\sc D. C. Sorensen}, {\em Newton's method with a model trust region modification}, SIAM J. Numer. Anal., 19(2) (1982), pp. 409--426.

\bibitem{M83}
{\sc J. J. Mor\'{e}, D. C. Sorensen}, {\em Computing a trust region step}, SIAM J. Sci. Statist. Comput., 4(3) (1983), pp. 553--572.

\bibitem{M94}
{\sc J. M. Mart\'{i}nez}, {\em Local minimizers of quadratic functions on Euclidean balls and spheres}, SIAM J. Optim., 4(1) (1994), pp. 159--176.

\bibitem{WX20}
{\sc J. Wang, Y. Xia}, {\em Closing the gap between necessary and sufficient conditions for local nonglobal minimizer of trust region subproblem}, SIAM J. Optim., 30(3) (2020), pp. 1980--1995.

\bibitem{LPR98}
{\sc S. Lucidi, L. Palagi, M. Roma}, {\em On some properties of quadratic programs with a convex quadratic constraint}, SIAM J. Optim., 8(1) (1998), pp. 105--122.

\bibitem{G80}
{\sc W. Gander}, {\em Least squares with a quadratic constraint}, Numer. Math., 36(3) (1980), pp. 291--307.

\bibitem{C85}
{\sc M. R. Celis, J. E. Dennis, R. A. Tapia}, {\em A trust region strategy for nonlinear inequality constrained optimization}, in Numerical Optimization, R.T. Boggs, R.H. Byrd, and R.B.
Schnabel, eds. (1984), SIAM, Philadelphia, PA, (1985), pp. 71--82.

\bibitem{CY00}
{\sc X. Chen, Y. Yuan}, {\em On local solutions of the  {Celis-Dennis-Tapia} subproblem}, SIAM J. Optim., 10(2) (2000), pp. 359--383.

\bibitem{GRT10}
{\sc N. I. M. Gould, D. P. Robinson, H. Sue Thorne}, {\em On solving trust-region and other regularised subproblems in optimization}, Math. Program. Comput., 2(1) (2010), pp. 21--57.

\bibitem{Hsia17}
{\sc Y. Hsia, R. L. Sheu, Y. Yuan}, {\em Theory and application of $p$-regularized subproblems for $p>2$}, Optim. Method Softw., 32(5) (2017), pp. 1059-1077.


\bibitem{N06}
{\sc Y. Nesterov, B. T. Polyak}, {\em Cubic regularization of Newton method and its global performance}, Math. Program., 108(1) (2006), pp. 177--205.

\bibitem{W07}
{\sc M. Weiser, P. Deuflhard, B. Erdmann}, {\em Affine conjugate adaptive Newton methods for
nonlinear elastomechanics}, Optim. Methods Softw., 22(3) (2007), pp. 413--431.

\bibitem{C11}
{\sc C. Cartis, N. I. M. Gould, Ph. L. Toint}, {\em Adaptive cubic regularisation methods for unconstrained optimization. Part I: motivation, convergence and numerical results}, Math. Program., 127(2) (2011), pp. 245--295.

\bibitem{CGT11}
{\sc C. Cartis, N. I. M. Gould, Ph. L. Toint}, {\em Adaptive cubic regularisation methods for unconstrained optimization. Part II: worst-case function-and derivative-evaluation complexity}, Math. Program., 130(2) (2011), pp. 295--319.

\bibitem{L20}
{\sc F. Lieder}, {\em Solving large-Scale cubic regularization by a generalized eigenvalue problem}, SIAM J. Optim., 30(4) 2020, pp. 3345--3358.


\bibitem{F17}
{\sc S. C. Fang, D. Gao, G. X. Lin, R. L. Sheu, W. Xing, }, {\em Double well potential function and its optimization in the n-dimensional real space: part I}, J. Ind. Manag. Optim., 13(3) (2017), pp.  1291--1305.

\bibitem{X17}
{\sc Y. Xia, R. L. Sheu, S. C. Fang, W. Xing}, {\em
Double well potential function and its optimization in the n-dimensional real space: part II}, J. Ind. Manag. Optim., 13(3) (2017), pp. 1307--1328.

\bibitem{A17}
{\sc S. Adachi, S. Iwata, Y. Nakatsukasa, A. Takeda}, {\em Solving the trust-region subproblem by a generalized eigenvalue problem}, SIAM J. Optim., 27(1) (2017), pp. 269--291.

\bibitem{S17}
{\sc M. Salahi, A. Taati, H. Wolkowicz}, {\em Local nonglobal minima for solving large-scale extended trust-region subproblems}, Comput. Optim. Appl., 66(2) (2017), pp. 223--244.

\end{thebibliography}
\end{document}